\documentclass[reqno]{amsart}
\usepackage{amsmath}

\begin{document}
\title[singular solutions for a 2x2 nonconservative system]
{Singular solutions for a 2x2 system in nonconservative form with incomplete set of eigenvectors}

\author[Anupam Pal Choudhury]
{Anupam Pal Choudhury}

\address{Anupam Pal Choudhury \newline
TIFR Centre for Applicable Mathematics\\
Sharada Nagar, Chikkabommasandra, GKVK P.O.\\
Bangalore 560065, India}
\email{anupam@math.tifrbng.res.in}


\begin{abstract}
In this paper, we study the initial-value problem for two first order systems in non-conservative form.
The first system arises in elastodynamics and belongs to the class of strictly hyperbolic, genuinely nonlinear systems. The second system
has repeated eigenvalues and an incomplete set of right eigenvectors. Solutions to such systems are expected to develop singular concentrations.
Existence of singular solutions to both the systems have been shown using the method of weak asymptotics. The second system has been shown to develop
singular concentrations even from Riemann-type initial data. The first system differing from the second
in having an extra term containing a positive constant k, the solution constructed for the first system have been shown to converge to the solution
of the second as k tends to 0. 
\end{abstract}

\maketitle
\numberwithin{equation}{section}
\numberwithin{equation}{section}
\newtheorem{theorem}{Theorem}[section]
\newtheorem{remark}[theorem]{Remark}

\section{Introduction}
The initial-value problem for the first-order quasilinear hyperbolic system 
\begin{equation}
\begin{aligned} 
 &\frac{\partial u}{\partial t}+u\frac{\partial u}{\partial x}-\frac{\partial \sigma}{\partial x}=0, \\
 &\frac{\partial \sigma}{\partial t}+u\frac{\partial \sigma}{\partial x}-k^{2}\frac{\partial u}{\partial x}=0
 \end{aligned}
\label{e1.1}
\end{equation}
 (in the domain $\Omega=\{(x,t):-\infty<x<\infty,t>0\}$) arising in applications in elastodynamics, has been well studied (see \cite{c1},\cite{j1},\cite{j2},\cite{l1}). Here $k$ is a positive constant. It is a strictly hyperbolic system having two real distinct eigenvalues given by
$$\lambda_{1}(u,\sigma)=u-k,\ \lambda_{2}(u,\sigma)=u+k$$
with the corresponding right eigenvectors
$$E_{1}(u,\sigma)=\left( \begin{array}{c}
1 \\
k \end{array} \right),\ 
E_{2}(u,\sigma)=\left(\begin{array}{c}
                       1 \\
                       -k \end{array}\right).$$
Now letting $k \rightarrow 0$, we see that the eigenvalues $\lambda_{1}(u,\sigma)$ and $\lambda_{2}(u,\sigma)$ tend to coincide. In particular, taking $k=0$ in \eqref{e1.1} we arrive at the system 
\begin{equation}
\begin{aligned}
 &\frac{\partial u}{\partial t}+u\frac{\partial u}{\partial x}-\frac{\partial \sigma}{\partial x}=0, \\
 &\frac{\partial \sigma}{\partial t}+u\frac{\partial \sigma}{\partial x}=0.
\end{aligned}
\label{e1.2}
\end{equation}
which has repeated eigenvalues $\lambda_{1}(u,\sigma)=\lambda_{2}(u,\sigma)=u$ and an incomplete set of right eigenvectors (we can take 
$\left(\begin{array}{c} 
                       1 \\
                       0\end{array}\right)$ to be a right eigenvector).\\
In \cite{z1}, a class of 2x2 systems in conservative form having an incomplete set of eigenvectors everywhere has been considered. These systems exhibit
development of singular concentrations. We expect a similar kind of development of singular concentration for the system \eqref{e1.2}. But the analysis in \cite{z1} cannot be applied directly due to the following two reasons:\\
1. The system \eqref{e1.2} is in nonconservative form and hence we need to give a suitable meaning to the nonconservative products which in general lead to different solutions depending upon the meaning attached (see \cite{d1},\cite{r1},\cite{v1}).\\
2. One of the assumptions that has been used in \cite{z1} is that the eigenvalue should have vanishing directional derivative along a right eigenvector. But the eigenvalue $u$ in this case has a nonvanishing directional derivative along any right eigenvector. Thus the assumption is not satisfied.\\

Another reason to expect singular solutions for the system \eqref{e1.2} comes from studying the behaviour as $k\rightarrow 0$ of the 
shock and rarefaction curves obtained for the Riemann problem for the system \eqref{e1.1}. In \cite{j2}, the Riemann problem for the system \eqref{e1.1} has been studied using Volpert's product. Starting with Riemann type initial data $$(u(x,0),\sigma(x,0))=\begin{cases} (u_{L},\sigma_{L}),\,\,\ x<0\\
                        (u_{R},\sigma_{R}),\,\,\ x>0.
                       \end{cases}$$
the shock curves $S_{1}(u_{L},\sigma_{L}),S_{2}(u_{L},\sigma_{L})$ and the rarefaction curves $R_{1}(u_{L},\sigma_{L}),$\\
$R_{2}(u_{L},\sigma_{L})$ can be written down in the $u-\sigma$ plane as in \cite{j2}:
\begin{equation}
 \begin{aligned}
  &R_{1}(u_{L},\sigma_{L}):\ \sigma=\sigma_{L}+k(u-u_{L}),\ \ \ u>u_{L},\\
  &R_{2}(u_{L},\sigma_{L}):\ \sigma=\sigma_{L}-k(u-u_{L}),\ \ \ u>u_{L},\\
  &S_{1}(u_{L},\sigma_{L}):\ \sigma=\sigma_{L}+k(u-u_{L}),\ \ \ u<u_{L},\\
  &S_{2}(u_{L},\sigma_{L}):\ \sigma=\sigma_{L}-k(u-u_{L}),\ \ \ u<u_{L}.
 \end{aligned}
\label{e1.3}
\end{equation}
Now as $k \rightarrow 0$, it seems from the above curves that we cannot have a jump in $\sigma$ and therefore the Riemann problem might not be solvable
using Volpert's product. That this is indeed true when $k=0$, that is, for the system \eqref{e1.2} follows from the following:\\
\\
For the solution of the Riemann problem using Volpert's product the following relations have to be satisfied (see \cite{d1})
$$-s[u]+[\frac{u^2}{2}]-[\sigma]=0,\\
  -s[\sigma]+\frac{1}{2}(u_{L}+u_{R})[\sigma]=0$$ 
where $s$ denotes the speed of the discontinuity and $[w]$ denotes the jump in the function $w$ across the discontinuity. It follows from the above relations that $[\sigma]=0$ thus proving what we expected.\\
\\
This observation again makes us suspect that there might be development of singular concentrations even if we start with Riemann type initial data.\\
\\
In what follows our main aim is to construct \textit{generalised $\delta-$shock wave solutions} for the systems \eqref{e1.1} and \eqref{e1.2} and examine the role played by the constant $k$.
We use the \textit{method of weak asymptotics} (see \cite{a1},\cite{d2},\cite{k1}) as the main tool. Depending upon the cases, we shall sometimes use the complex-valued \textit{weak asymptotic solutions} (see \cite{k1}) for our construction.\\
\\
The plan of the paper is as follows: In Section 2, the notions of the weak asymptotic solutions, generalised $\delta-$shock wave type solutions are recalled and a brief sketch
of the method of weak asymptotics is given. In Section 3, a few weak asymptotic expansions are proved which are crucially used in the construction of singular solutions. 
In Section 4, existence of generalised $\delta-$shock wave type solutions for the systems \eqref{e1.1} and \eqref{e1.2}
is proved. The role played by $k$ is examined and the development of singular solutions from Riemann type initial data is exhibited. 

\section{A brief discussion on the method of weak asymptotics}
In this section, we recall the definitions of weak asymptotic solutions and generalised $\delta-$shock wave type solutions and discuss
the method of weak asymptotics (see \cite{a1},\cite{d2},\cite{k1}).\\  
Let $\mathcal{D}$ and $\mathcal{D}^{\prime}$ denote the space of smooth functions of compact support and the space of distributions respectively.\\
\\
Let $O_{\mathcal{D}^{\prime}}(\epsilon^{\alpha})$ denote the collection of distributions $f(x,t,\epsilon) \in \mathcal{D}^{\prime}(\mathbb{R})$ such that for any test function
$\varphi(x) \in \mathcal{D}(\mathbb{R})$ the estimate $$\langle f(x,t,\epsilon),\varphi(x)\rangle=O(\epsilon^{\alpha})$$
holds and is uniform with respect to $t$. The relation $o_{\mathcal{D}^{\prime}}(\epsilon^{\alpha})$ is interpreted similarly.\\
\\
\textbf{Definition 2.1}.(\cite{k1}) A pair of smooth complex-valued (real-valued) functions \\
$(u(x,t,\epsilon),\sigma(x,t,\epsilon))$ is called a \textit{weak asymptotic solution}
of the system \eqref{e1.1} with the initial data $(u(x,0),\sigma(x,0))$ if
\begin{equation}
 \begin{aligned}
  &\frac{\partial u(x,t,\epsilon)}{\partial t}+u(x,t,\epsilon)\frac{\partial u(x,t,\epsilon)}{\partial x}-\frac{\partial \sigma(x,t,\epsilon)}{\partial x}=o_{\mathcal{D}^{\prime}}(1), \\
  &\frac{\partial \sigma(x,t,\epsilon)}{\partial t}+u(x,t,\epsilon)\frac{\partial \sigma(x,t,\epsilon)}{\partial x}-k^{2}\frac{\partial u(x,t,\epsilon)}{\partial x}=o_{\mathcal{D}^{\prime}}(1),\\
  &u(x,0,\epsilon)-u(x,0)=o_{\mathcal{D}^{\prime}}(1),\\
  &\sigma(x,0,\epsilon)-\sigma(x,0)=o_{\mathcal{D}^{\prime}}(1),\ \epsilon \rightarrow 0.
   \end{aligned}
 \label{e2.1}
\end{equation}
\\
Similarly, for the system \eqref{e1.2} we define\\
\textbf{Definition 2.2}.(\cite{k1}) A pair of smooth complex-valued (real-valued) functions \\
$(u(x,t,\epsilon),\sigma(x,t,\epsilon))$ is called a \textit{weak asymptotic solution}
of the system \eqref{e1.2} with the initial data $(u(x,0),\sigma(x,0))$ if
\begin{equation}
 \begin{aligned}
  &\frac{\partial u(x,t,\epsilon)}{\partial t}+u(x,t,\epsilon)\frac{\partial u(x,t,\epsilon)}{\partial x}-\frac{\partial \sigma(x,t,\epsilon)}{\partial x}=o_{\mathcal{D}^{\prime}}(1), \\
  &\frac{\partial \sigma(x,t,\epsilon)}{\partial t}+u(x,t,\epsilon)\frac{\partial \sigma(x,t,\epsilon)}{\partial x}=o_{\mathcal{D}^{\prime}}(1),\\
  &u(x,0,\epsilon)-u(x,0)=o_{\mathcal{D}^{\prime}}(1),\\
  &\sigma(x,0,\epsilon)-\sigma(x,0)=o_{\mathcal{D}^{\prime}}(1),\ \epsilon \rightarrow 0.
   \end{aligned}
 \label{e2.2}
\end{equation}
\\
\textbf{Definition 2.3}. A pair of real-valued distributions $(u(x,t),\sigma(x,t))\in C(\mathbb{R}_{+};\mathcal{D}^{\prime}(\mathbb{R}))$
is called a \textit{generalised solution} of the systems \eqref{e1.1} or \eqref{e1.2} if it is the weak limit
(limit in the sense of distributions) of a weak asymptotic solution $(u(x,t,\epsilon),\sigma(x,t,\epsilon))$ as $\epsilon \rightarrow 0$.\\

\subsection{The method of weak asymptotics} Instead of writing it down schematically (see \cite{a1}), here we shall try to illustrate the method of weak asymptotics by applying it to 
the systems \eqref{e1.1} or \eqref{e1.2} with initial data of the form
\begin{equation}
 \begin{aligned}
  &u(x,0)=u_{0}+u_{1}H(-x),\\
  &\sigma(x,0)=\sigma_{0}+\sigma_{1}H(-x)+e^{0}\delta(x).
 \end{aligned}
 \label{e2.3}
 \end{equation}
where $u_{0},u_{1},\sigma_{0},\sigma_{1},e^{0}$ are constants. In this case we seek solutions in the form of the singular ansatz
\begin{equation}
 \begin{aligned}
  &u(x,t)=u_{0}+u_{1}H(-x+\phi(t)),\\
  &\sigma(x,t)=\sigma_{0}+\sigma_{1}H(-x+\phi(t))+e(t)\delta(x-\phi(t))
 \end{aligned}
\label{e2.4}
\end{equation}
Such a solution will be called a \textit{generalised $\delta-$shock wave type solution} for the systems \eqref{e1.1} or \eqref{e1.2}.\\
The first step is to regularize the Heaviside functions and delta distributions and add \textit{correction} terms so as to form a 
smooth ansatz for the weak asymptotic solutions in the form:
\begin{equation}
 \begin{aligned}
 & u(x,t,\epsilon)=u_{0}+u_{1}H_{u}(-x+\phi(t),\epsilon)+R_{u}(x,t,\epsilon),\\
 & \sigma(x,t,\epsilon)=\sigma_{0}+\sigma_{1}H_{\sigma}(-x+\phi(t),\epsilon)+e(t)\delta(x-\phi(t),\epsilon)+R_{\sigma}(x,t,\epsilon).
 \end{aligned}
 \label{e2.5}
\end{equation}
Here $H_{u},H_{\sigma}$ are regularizations of the Heaviside function, $\delta(x,\epsilon)$ is a regularization of the delta function and
$R_{u},R_{\sigma}$ are \textit{correction} terms chosen so as to satisfy the conditions:
$$R_{i}(x,t,\epsilon)=o_{\mathcal{D}^{\prime}}(1),\ \frac{\partial R_{i}(x,t,\epsilon)}{\partial t}=o_{\mathcal{D}^{\prime}}(1),\ 
\epsilon \rightarrow 0,\ i=u,\sigma.$$
The next step is to prove the existence of a weak asymptotic solution by substituting the smooth ansatz in place of u and $\sigma$ in the left hand side
of the systems and finding out the functions $\phi(t)$ and $e(t)$.\\
Once this step is completed, we obtain a generalised $\delta-$shock wave type solution of the form \eqref{e2.4} as the distributional limit
of the weak asymptotic solution as $\epsilon \rightarrow 0$.

\section{The regularizations and weak asymptotic expansions}
In this section, we describe the regularizations and correction term to be used in the next section and prove a few weak asymptotic expansions for
various products containing them.\\
As mentioned in Section 2, we seek \textit{generalised $\delta-$shock wave type solutions} of the form
\begin{equation}
 \begin{aligned}
  &u(x,t)=u_{0}+u_{1}H(-x+\phi(t)),\\
  &\sigma(x,t)=\sigma_{0}+\sigma_{1}H(-x+\phi(t))+e(t)\delta(x-\phi(t))
 \end{aligned}
 \label{e3.1}
\end{equation}
 for the systems \eqref{e1.1} and \eqref{e1.2} with initial data of the form
 \begin{equation}
 \begin{aligned}
  &u(x,0)=u_{0}+u_{1}H(-x),\\
  &\sigma(x,0)=\sigma_{0}+\sigma_{1}H(-x)+e^{0}\delta(x).
 \end{aligned}
 \label{e3.2}
 \end{equation}
Here $\phi(t)$ and $e(t)$ are smooth functions to be determined.
To apply the method of weak asymptotics, we suggest a smooth ansatz of the form (we do not take any correction term in the expression for $\sigma(x,t,\epsilon)$)
\begin{equation}
 \begin{aligned}
 & u(x,t,\epsilon)=u_{0}+u_{1}H_{u}(-x+\phi(t),\epsilon)+p(t)R(x-\phi(t),\epsilon),\\
 & \sigma(x,t,\epsilon)=\sigma_{0}+\sigma_{1}H_{\sigma}(-x+\phi(t),\epsilon)+e(t)\delta(x-\phi(t),\epsilon).
 \end{aligned}
 \label{e3.3}
\end{equation}
Here $H_{u}(.,\epsilon),H_{\sigma}(.,\epsilon)$ are regularizations of the Heaviside function, $\delta(.,\epsilon)$ is the regularization
of the delta distribution and $p(t)R(.,\epsilon)$ is a correction term where $p(t)$ is a smooth function (complex-valued or real-valued)
to be chosen afterwards.\\
The regularizations of the delta distribution and the correction term used in the proofs are same as in \cite{k1}. The regularization of
the Heaviside function has a subtle change though.\\

Let $\omega:\mathbb{R}\rightarrow \mathbb{R}$ be a non-negative, smooth, even function with support in $(-1,1)$ and satisfying
$$\int_{\mathbb{R}}\omega(x)dx=1.$$
Let $\omega_{0}=\int_{\mathbb{R}}\omega^{2}(x)dx$ and let 
\begin{equation}
 R(x,t,\epsilon)=\frac{1}{\sqrt{\epsilon}}\omega(\frac{x-2\epsilon}{\epsilon}),\ \
\delta(x,\epsilon)=\frac{1}{\epsilon}\omega(\frac{x+2\epsilon}{\epsilon}).
\label{e3.4}
\end{equation}
Since by definition, $R$ is independent of $t$, we henceforth denote it by $R(x,\epsilon)$.\\
An easy calculation using test functions then shows that $R(x,\epsilon)=o_{\mathcal{D}^{\prime}}(\epsilon)$ and since $R$ is independent of t, we also have
$\frac{\partial R(x,\epsilon)}{\partial t}=o_{\mathcal{D}^{\prime}}(\epsilon)$. Thus it satisfies the criteria for being a \textit{correction}
term as stated in Section 2.\\
Also it can be easily checked that $$\frac{\partial R(x,\epsilon)}{\partial x}=o_{\mathcal{D}^{\prime}}(\epsilon),\ R^{2}(x,\epsilon)=\omega_{0}\delta(x)+o_{\mathcal{D}^{\prime}}(\epsilon),\ R(x,\epsilon)\frac{\partial R(x,\epsilon)}{\partial x}=\frac{1}{2}\omega_{0}\delta^{\prime}(x)+o_{\mathcal{D}^{\prime}}(\epsilon)$$
and that $$\delta(x,\epsilon)=\delta(x)+o_{\mathcal{D}^{\prime}}(\epsilon),\ \ \frac{\partial \delta(x,\epsilon)}{\partial x}=\delta^{\prime}(x)+o_{\mathcal{D}^{\prime}}(\epsilon).$$
An important observation we need to make at this point is that the supports of $R(x,\epsilon)$ and $\delta(x,\epsilon)$ are disjoint. Therefore
we have $$R(x,\epsilon)\delta(x,\epsilon)=0,\ R(x,\epsilon)\frac{\partial \delta(x,\epsilon)}{\partial x}=0.$$
Let $c=(\frac{1}{2}-\frac{\sigma_{1}}{u_{1}^{2}})$.\\
Let us define the smooth function $H(x,\epsilon)$ as follows: 
\begin{equation}
H(x,\epsilon)=\begin{cases} 1,\ \ x\leq -4\epsilon\\
                           c,\ \ -3\epsilon\leq x \leq 3\epsilon\\
                           0,\ \ x\geq 4\epsilon
                           \end{cases}
\label{e3.5}
\end{equation}
and is continued smoothly in the regions $(-4\epsilon,-3\epsilon)$ and $(3\epsilon,4\epsilon)$.\\
We take $H_{u}(x,\epsilon)=H_{\sigma}(x,\epsilon)=H(x,\epsilon)$. Again a little bit of calculation shows that
$$H(x,\epsilon)=H(x)+o_{\mathcal{D}^{\prime}}(\epsilon),\ \frac{\partial H(x,\epsilon)}{\partial x}=\delta(x)+o_{\mathcal{D}^{\prime}}(\epsilon),\ 
H(x,\epsilon)\frac{\partial H(x,\epsilon)}{\partial x}=\frac{1}{2}\delta(x)+o_{\mathcal{D}^{\prime}}(\epsilon)$$ 
Since the supports of $R(x,\epsilon)$ and $\delta(x,\epsilon)$ are contained in $(-3\epsilon,3\epsilon)$, it again follows that
\begin{equation}
 \begin{aligned}
  &H(x,\epsilon)\frac{\partial R(x,\epsilon)}{\partial x}=c\frac{\partial R(x,\epsilon)}{\partial x}=o_{\mathcal{D}^{\prime}}(\epsilon),\\
  &R(x,\epsilon)\frac{\partial H(x,\epsilon)}{\partial x}=0.R(x,t,\epsilon)=o_{\mathcal{D}^{\prime}}(\epsilon),\\
  &H(x,\epsilon)\frac{\partial \delta(x,\epsilon)}{\partial x}=c\delta^{\prime}(x)+o_{\mathcal{D}^{\prime}}(\epsilon).
  \end{aligned}
\notag
\end{equation}
From the above discussions we then have the following lemma\\
\\
\textbf{Lemma 3.1}.\textit{ Choosing the regularizations and corrections as in \eqref{e3.4} and \eqref{e3.5} we have the following weak
asymptotic expansions:
\begin{equation}
 \begin{aligned}
 &R(x,\epsilon)=o_{\mathcal{D}^{\prime}}(1),\ \ \frac{\partial R(x,\epsilon)}{\partial x}=o_{\mathcal{D}^{\prime}}(1),\\
 &R^{2}(x,\epsilon)=\omega_{0}\delta(x)+o_{\mathcal{D}^{\prime}}(1),\\
 &R(x,\epsilon)\frac{\partial R(x,\epsilon)}{\partial x}=\frac{1}{2}\omega_{0}\delta^{\prime}(x)+o_{\mathcal{D}^{\prime}}(1),\\
 &\delta(x,\epsilon)=\delta(x)+o_{\mathcal{D}^{\prime}}(1),\ \ \frac{\partial \delta(x,\epsilon)}{\partial x}=\delta^{\prime}(x)+o_{\mathcal{D}^{\prime}}(1),\\
 &R(x,\epsilon)\delta(x,\epsilon)=0,\ R(x,\epsilon)\frac{\partial \delta(x,\epsilon)}{\partial x}=0,\\
 &H(x,\epsilon)=H(x)+o_{\mathcal{D}^{\prime}}(1),\ \frac{\partial H(x,\epsilon)}{\partial x}=\delta(x)+o_{\mathcal{D}^{\prime}}(1),\\ 
 &H(x,\epsilon)\frac{\partial H(x,\epsilon)}{\partial x}=\frac{1}{2}\delta(x)+o_{\mathcal{D}^{\prime}}(1),\\
 &H(x,\epsilon)\frac{\partial R(x,\epsilon)}{\partial x}=o_{\mathcal{D}^{\prime}}(1),\ R(x,\epsilon)\frac{\partial H(x,\epsilon)}{\partial x}=o_{\mathcal{D}^{\prime}}(1),\\
 &H(x,\epsilon)\frac{\partial \delta(x,\epsilon)}{\partial x}=c\delta^{\prime}(x)+o_{\mathcal{D}^{\prime}}(1),\ \epsilon \rightarrow 0.
 \end{aligned}
 \end{equation}}

\section{Generalised Delta-shock wave type solutions via construction of weak asymptotic solutions}
In this section, we construct $\delta-$shock wave type solutions for the systems \eqref{e1.1} and \eqref{e1.2} using the method of weak 
asymptotics.\\
For the rest of the discussion, we use the convention $[v]=v_{L}-v_{R}$, where $v_{L},v_{R}$ respectively denote the left and right states of $v$ across the discontinuity.
\begin{theorem}
 For $t\in [0,\infty)$, the Cauchy problem \eqref{e1.1},\eqref{e3.2} has a weak asymptotic solution \eqref{e3.3} with $\phi(t),e(t)\ and\ p(t)$
given by the relations 
\begin{equation}
\begin{aligned}
&\dot{\phi}(t)=\frac{[\frac{u^{2}}{2}]-[\sigma]}{[u]},\  \dot{e}(t)=\frac{\sigma_{1}^{2}}{u_{1}}-k^{2}u_{1},\\
&\frac{1}{2}p^{2}(t)\omega_{0}-e(t)=0,
\end{aligned}
\label{e4.1}
\end{equation}
where $\omega_{0}$ is a positive constant $($defined in Section $3$ $)$.
\end{theorem}
 
\begin{proof}
 To begin with, let's recall the form of the smooth ansatz:
 \begin{align}
 & u(x,t,\epsilon)=u_{0}+u_{1}H_{u}(-x+\phi(t),\epsilon)+p(t)R(x-\phi(t),\epsilon),\notag \\
 & \sigma(x,t,\epsilon)=\sigma_{0}+\sigma_{1}H_{\sigma}(-x+\phi(t),\epsilon)+e(t)\delta(x-\phi(t),\epsilon).\notag
 \end{align}
Then the partial derivatives of $u(x,t,\epsilon)$ and $\sigma(x,t,\epsilon)$ are given by
\begin{equation}
\begin{aligned}
  \frac{\partial u(x,t,\epsilon)}{\partial t} & =u_{1}\dot{\phi}(t)\frac{dH_{u}(-x+\phi(t),\epsilon)}{d\xi}+\dot{p}(t)R(x-\phi(t),\epsilon)-p(t)\dot{\phi}(t)\frac{dR(x-\phi(t),\epsilon)}{d\xi},     \notag \\
  \frac{\partial u(x,t,\epsilon)}{\partial x} &=-u_{1}\frac{dH_{u}(-x+\phi(t),\epsilon)}{d\xi}+p(t)\frac{dR(x-\phi(t),\epsilon)}{d\xi}, \notag \\
  \frac{\partial \sigma(x,t,\epsilon)}{\partial t} &=\sigma_{1}\dot{\phi}(t)\frac{dH_{\sigma}(-x+\phi(t),\epsilon)}{d\xi}+\dot{e}(t)\delta(x-\phi(t),\epsilon)-e(t)\dot{\phi}(t)\frac{d\delta(x-\phi(t),\epsilon)}{d\xi},   \notag \\
  \frac{\partial \sigma(x,t,\epsilon)}{\partial x} &=-\sigma_{1}\frac{dH_{\sigma}(-x+\phi(t),\epsilon)}{d\xi}+e(t)\frac{d\delta(x-\phi(t),\epsilon)}{d\xi},
 \end{aligned}
\end{equation}
where $\frac{du(.,\epsilon)}{d\xi}$ denotes the derivative of $u(.,\epsilon)$ with respect to the first component.\\ 
Substituting these relations into the left-hand side of the system \eqref{e1.1}, we have the following
\begin{equation}
 \begin{aligned}
  &\frac{\partial u(x,t,\epsilon)}{\partial t} +u(x,t,\epsilon)\frac{\partial u(x,t,\epsilon)}{\partial x}-\frac{\partial \sigma(x,t,\epsilon)}{\partial x}   \\
  &=\ u_{1}\dot{\phi}(t)\frac{dH_{u}(-x+\phi(t),\epsilon)}{d\xi} +\dot{p}(t)R(x-\phi(t),\epsilon)-p(t)\dot{\phi}(t)\frac{dR(x-\phi(t),\epsilon)}{d\xi}   \\
  &-u_{0}u_{1}\frac{dH_{u}(-x+\phi(t),\epsilon)}{d\xi} -u_{1}^{2}H_{u}(-x+\phi(t),\epsilon)\frac{dH_{u}(-x+\phi(t),\epsilon)}{d \xi}   \\
  & +u_{0} p(t)\frac{dR(x-\phi(t),\epsilon)}{d\xi} +u_{1}p(t)H_{u}(-x+\phi(t),\epsilon)\frac{dR(x-\phi(t),\epsilon)}{d \xi}    \\
  &-p(t)u_{1}R(x-\phi(t),\epsilon)\frac{dH_{u}(-x+\phi(t),\epsilon)}{d \xi} +p^{2}(t)R(x-\phi(t),\epsilon)\frac{dR(x-\phi(t),\epsilon)}{d \xi}    \\
  & +\sigma_{1}\frac{dH_{\sigma}(-x+\phi(t),\epsilon)}{d \xi}-e(t)\frac{d\delta(x-\phi(t),\epsilon)}{d\xi}
 \end{aligned}
\label{e4.2}
\end{equation}
 and
 \begin{equation}
 \begin{aligned} 
  &\frac{\partial \sigma(x,t,\epsilon)}{\partial t}+u(x,t,\epsilon)\frac{\partial \sigma(x,t,\epsilon)}{\partial x}-k^{2}\frac{\partial u(x,t,\epsilon)}{\partial x}  \\
  &=\ \sigma_{1}\dot{\phi}(t)\frac{dH_{\sigma}(-x+\phi(t),\epsilon)}{d\xi}+\dot{e}(t)\delta(x-\phi(t),\epsilon)
  -e(t)\dot{\phi}(t)\frac{d\delta(x-\phi(t),\epsilon)}{d\xi} \\
  &-u_{0}\sigma_{1}\frac{dH_{\sigma}(-x+\phi(t),\epsilon)}{d \xi}-u_{1}\sigma_{1}H_{u}(-x+\phi(t),\epsilon)\frac{dH_{\sigma}(-x+\phi(t),\epsilon)}{d \xi} \\
  &+u_{0}e(t)\frac{d\delta(x-\phi(t),\epsilon)}{d\xi}+u_{1}e(t)H_{u}(-x+\phi(t),\epsilon)\frac{d\delta(x-\phi(t),\epsilon)}{d\xi}  \\
  &-\sigma_{1}p(t)R(x-\phi(t),\epsilon)\frac{dH_{\sigma}(-x+\phi(t),\epsilon)}{d \xi}+p(t)e(t)R(x-\phi(t),\epsilon)\frac{d\delta(x-\phi(t),\epsilon)}{d\xi}  \\
  &+k^{2}u_{1}\frac{dH_{u}(-x+\phi(t),\epsilon)}{d\xi}-k^{2}p(t)\frac{dR(x-\phi(t),\epsilon)}{d\xi}.
\end{aligned}
\label{e4.3}
\end{equation}
Now using the weak asymptotics relations from Lemma 3.1 in the relation \eqref{e4.2}, we obtain
\begin{equation}
 \begin{aligned}
  &\frac{\partial u(x,t,\epsilon)}{\partial t}+u(x,t,\epsilon)\frac{\partial u(x,t,\epsilon)}{\partial x}-\frac{\partial \sigma(x,t,\epsilon)}{\partial x} \\
  &=\ \{u_{1}\dot{\phi}(t)-u_{0}u_{1}-\frac{1}{2}u_{1}^{2}+\sigma_{1}\}\delta(x-\phi(t))+\{\frac{1}{2}p^{2}(t)\omega_{0}-e(t)\}\dot{\delta}(x-\phi(t))\\
  &+o_{\mathcal{D}^{\prime}}(1).
\end{aligned}
\notag
\end{equation}
Setting the coefficients of $\delta$ and $\dot{\delta}$ in the above relation to be zero, we obtain
\begin{equation}
 \begin{aligned}
  &u_{1}\dot{\phi}(t)-u_{0}u_{1}-\frac{1}{2}u_{1}^{2}+\sigma_{1}=0, \\
  &\frac{1}{2}p^{2}(t)\omega_{0}-e(t)=0.
   \end{aligned}
\notag
\end{equation}
The first equation above, when rewritten, gives
\begin{equation}
 \dot{\phi}(t)=\frac{[\frac{u^{2}}{2}]-[\sigma]}{[u]}.
\notag
\end{equation}
Substituting $\dot{\phi}(t)$ from above in the relation \eqref{e4.3} and observing that
\begin{equation}
\begin{aligned}
 u_{1}e(t)H_{u}(-x+\phi(t),\epsilon)\frac{d\delta(x-\phi(t),\epsilon)}{d\xi}&+u_{0}e(t)\frac{d\delta(x-\phi(t),\epsilon)}{d\xi}\\
&-e(t)\dot{\phi}(t)\frac{d\delta(x-\phi(t),\epsilon)}{d\xi}=o_{\mathcal{D}^{\prime}}(\epsilon),
\end{aligned}
\notag
\end{equation}
(the choice of $c$ as in Section 3 helps us in getting this asymptotics)
we have
\begin{equation}
 \begin{aligned}
  &\frac{\partial \sigma(x,t,\epsilon)}{\partial t}+u(x,t,\epsilon)\frac{\partial \sigma(x,t,\epsilon)}{\partial x}-k^{2}\frac{\partial u(x,t,\epsilon)}{\partial x} \\
  &=\{\dot{e}(t)+\sigma_{1}\dot{\phi}(t)-\frac{1}{2}u_{1}\sigma_{1}-u_{0}\sigma_{1}+k^{2}u_{1}\}\delta(x-\phi(t))
  +o_{\mathcal{D}^{\prime}}(1),\epsilon \rightarrow 0.
 \end{aligned}
 \notag
\end{equation}
Setting the coefficient of $\delta$ in the above relation to be zero, we obtain
\begin{equation}
  \dot{e}(t)+\sigma_{1}\dot{\phi}(t)-\frac{1}{2}u_{1}\sigma_{1}-u_{0}\sigma_{1}+k^{2}u_{1}=0, \\
\notag
\end{equation}
which on simplification gives
\begin{equation}
  \dot{e}(t)=\frac{\sigma_{1}^{2}}{u_{1}}-k^{2}u_{1}.
\notag
\end{equation}
Therefore we find that the smooth ansatz \eqref{e3.3} is a $\mathit weak\  asymptotic\ solution$
provided $p(t),\phi(t),e(t)$ can be solved from the following equations
\begin{equation}
 \begin{aligned}
  &\dot{\phi}(t)=\frac{[\frac{u^{2}}{2}]-[\sigma]}{[u]}, \\
  &\dot{e}(t)=\frac{\sigma_{1}^{2}}{u_{1}}-k^{2}u_{1},\\
  &\frac{1}{2}p^{2}(t)\omega_{0}-e(t)=0.
   \end{aligned}
\label{e4.4}
\end{equation}
The above ordinary differential equations can be solved with the initial conditions $\phi(0)=0$ and $e(0)=e^{0}$ and we have
\begin{equation}
  \phi(t)=\frac{[\frac{u^{2}}{2}]-[\sigma]}{[u]}t,\ \  e(t)=(\frac{\sigma_{1}^{2}}{u_{1}}-k^{2}u_{1})t+e^{0}.
\label{e4.5}
 \end{equation}
Next substituting $e(t)$ in the last equation of \eqref{e4.4}, we can solve for $p(t)$ taken in the form $p(t)=p_{1}(t)+ip_{2}(t)$ and hence
we have a $\mathit weak\ asymptotic \ solution$ of the system \eqref{e1.1}.
\end{proof}
Since a $\mathit \delta-shock\ wave\ type\ solution$ is the distributional limit of a $\mathit weak$ \\
 $\mathit asymptotic\ solution$, we have from the previous theorem
\begin{theorem}
 For $t\in [0,\infty)$, the Cauchy problem \eqref{e1.1},\eqref{e3.2} has a generalised $\delta-$shock wave type solution \eqref{e3.1} with $\phi(t)\ and \ e(t)$
given by the relations 
\begin{equation}
\phi(t)=\frac{[\frac{u^{2}}{2}]-[\sigma]}{[u]}t,\  e(t)=(\frac{\sigma_{1}^{2}}{u_{1}}-k^{2}u_{1})t+e^{0}.
\label{e4.6}
\end{equation}
\end{theorem}
\begin{remark}
 Let $e^{0}=0$. Now $\dot e(t)=0$ in \eqref{e4.1} would imply the existence of shock-wave solution $($in the class of bounded variation functions$)$ of the system \eqref{e1.1}.
But that implies $\sigma_{1}^2=k^2 u_{1}^2$. Imposing Lax's admissibility condition, which in this case is $u_{1}>0$, we get the shock
curves 
\begin{equation}
\begin{aligned}
&S_{1}:[\sigma]=k[u],\\
&S_{2}:[\sigma]=-k[u]. 
\end{aligned}
 \notag
\end{equation}
Also in this case, we have $u\frac{\partial \sigma}{\partial x}=\lim_{\epsilon \to 0}u(x,t,\epsilon)\frac{\partial \sigma(x,t,\epsilon)}{\partial x}=-\sigma_{1}(u_{0}+\frac{u_{1}}{2})\delta$,
which again is the Volpert's product $($the negative sign arises because of the convention on $[\sigma]$ $)$. Thus we recover the results proved in \cite{j2} for the shock-wave case $($see 
\eqref{e1.3}$)$.  
\end{remark}
\begin{remark}$($Overcompressivity condition for $\delta-$shock wave solutions$)$
 We recall that the overcompressivity condition $($see \cite{a1},\cite{k1}$)$ for the $\delta-$shock wave solutions for a $n\times n$ system is
$$\lambda_{k}(v_{R})<\dot \phi(t)<\lambda_{k}(v_{L}),\ \ k=1,..,n.$$
Therefore for the system \eqref{e1.1}, it takes the form
\begin{equation}
 \begin{aligned}
 &u_{0}-k<\frac{[\frac{u^2}{2}]-[\sigma]}{[u]}<u_{0}+u_{1}-k,\\
 &u_{0}+k<\frac{[\frac{u^2}{2}]-[\sigma]}{[u]}<u_{0}+u_{1}+k. 
 \end{aligned}
\notag
\end{equation}
The above relations on simplification yield the conditions: $$u_{1}>2k>0$$ and $$-(\frac{u_{1}}{2}-k)<\frac{\sigma_{1}}{u_{1}}<\frac{u_{1}}{2}-k.$$\\
\end{remark}

Next we prove the existence of a $\mathit weak\ asymptotic\ solution$ for the system \eqref{e1.2}.
\begin{theorem}
For $t\in [0,\infty)$, the Cauchy problem \eqref{e1.2},\eqref{e3.2} has a weak asymptotic solution \eqref{e3.3} with $\phi(t),e(t)\ and\ p(t)$
given by the relations 
\begin{equation}
\begin{aligned}
&\dot{\phi}(t)=\frac{[\frac{u^{2}}{2}]-[\sigma]}{[u]},\  \dot{e}(t)=\frac{\sigma_{1}^{2}}{u_{1}},\\
&\frac{1}{2}p^{2}(t)\omega_{0}-e(t)=0,
\end{aligned}
\label{e4.7}
\end{equation}
where $\omega_{0}$ is a positive constant $($defined in Section $3$ $)$.
\end{theorem}
\begin{proof}
 Proceeding as in the proof of the Theorem 4.1,
we find that the smooth ansatz \eqref{e3.3} is a $\mathit weak\  asymptotic\ solution$
provided $p(t),\phi(t),e(t)$ can be solved from the following equations
\begin{equation}
 \begin{aligned}
  &\dot{\phi}(t)=\frac{[\frac{u^{2}}{2}]-[\sigma]}{[u]}, \\
  &\dot{e}(t)=\frac{\sigma_{1}^{2}}{u_{1}}, \\
  &\frac{1}{2}p^{2}(t)\omega_{0}-e(t)=0.
 \end{aligned}
\label{e4.8}
\end{equation}
The ordinary differential equations for $\phi(t)$ and $e(t)$ can be solved with the initial conditions $\phi(0)=0$ and $e(0)=e^{0}$ and we have
\begin{equation}
  \phi(t)=\frac{[\frac{u^{2}}{2}]-[\sigma]}{[u]}t,\ \  e(t)=\frac{\sigma_{1}^{2}}{u_{1}}t+e^{0}.
 \notag
 \end{equation}
Next substituting $e(t)$ in the last equation of \eqref{e4.8}, we can solve for $p(t)$ taken in the form $p(t)=p_{1}(t)+ip_{2}(t)$ and hence
we have a $\mathit weak\ asymptotic \ solution$ of the system \eqref{e1.2}.
\end{proof}
Therefore from the previous theorem, we have
\begin{theorem}
For $t\in [0,\infty)$, the Cauchy problem \eqref{e1.2},\eqref{e3.2} has a generalised $\delta-$shock wave type solution \eqref{e3.1} with $\phi(t)\ and \ e(t)$
given by the relations 
\begin{equation}
\phi(t)=\frac{[\frac{u^{2}}{2}]-[\sigma]}{[u]}t,\  e(t)=\frac{\sigma_{1}^{2}}{u_{1}}t+e^{0}.
\label{e4.9}
\end{equation}
\end{theorem}
\begin{remark}
 The overcompressivity assumption for the system \eqref{e1.2} yields $u_{1}>0$ and $-\frac{u_{1}}{2}<\frac{\sigma_{1}}{u_{1}}<\frac{u_{1}}{2}$.
\end{remark}
\begin{remark}
 If $u_{1}>0$ $($follows from the overcompressivity condition above$)$, then from equation \eqref{e4.7} it follows that $\dot{e}(t)>0$. If in addition
we have $e^{0}=0$, then $e(t)=\frac{\sigma_{1}^{2}}{u_{1}}t$ which is greater than zero for all t. In this case, it is sufficient to consider 
$p(t)$ as a real-valued function only.
\end{remark}

\begin{remark}
From Remark 4.8 it also follows that we might have a \textit{singular concentration} developing in the solution of the system \eqref{e1.2} even if we start with Riemann type initial data.\\
\end{remark}

\begin{remark}
 If we take $\sigma_{0}=\sigma_{1}=0$ in \eqref{e3.2}, then proceeding as in the proof of the Theorem 4.5 we obtain a generalised $\delta-$shock
 wave type solution for the system \eqref{e1.2} of the form:
\begin{equation}
 \begin{aligned}
  &u(x,t)=u_{0}+u_{1}H(-x+\phi(t)),\\
  &\sigma(x,t)=e^{0}\delta(x-\phi(t))
 \end{aligned}
 \end{equation}
where $\dot{\phi}(t)=\frac{[\frac{u^2}{2}]}{[u]}$.
\end{remark}

\begin{remark}$($Dependence of the solutions on k$)$\\
From the structure of the generalised solutions for the systems \eqref{e1.1},\eqref{e1.2} obtained from Theorem 4.2 and Theorem 4.6,
it's quite evident that as k tends to $0$, the generalised solution obtained for the system \eqref{e1.1} actually converges $($in distributional limit$)$ to that obtained 
for the system \eqref{e1.2}.\\
This observation therefore justifies our motivation to study the system \eqref{e1.2} based upon the solutions of the system \eqref{e1.1} $($letting $k\rightarrow 0$ $)$.
\end{remark}

\end{document}